\documentclass[11pt]{amsart}
\usepackage{latexsym}
\usepackage{graphicx} 
\usepackage[margin=0.8in]{geometry}
\usepackage{amsthm}
\usepackage{comment}
\usepackage{amssymb}
\usepackage{mathtools}
\usepackage{xcolor}
\usepackage{soul}
\sethlcolor{cyan}
\usepackage[style=alphabetic]{biblatex}
\DeclareMathOperator{\Hom}{Hom}

\newtheorem{proposition}{Proposition}
\newtheorem{lemma}{Lemma}
\newtheorem{definition}{Definition}
\newtheorem{theorem}{Theorem}
\newtheorem*{fact}{Fact}
\newtheorem{remark}{Remark}

\newcommand{\Z}{\mathbb{Z}}
\newcommand{\Q}{\mathbb{Q}}

\newcommand{\rank}{\mathrm{rank}}

\AtEndDocument{%
  \par
  \medskip
  \begin{tabular}{@{}l@{}}%
  
   Vlad Roman, Department of Mathematics, University of Southern California, \\
  Los Angeles, CA 90089-2532, USA \\
    \textit{E-mail address}: \texttt{vladroma@usc.edu} \\
  
  Robert M. Guralnick, Department of Mathematics, University of Southern California, \\
  Los Angeles, CA 90089-2532, USA \\
    \textit{E-mail address}: \texttt{guralnic@usc.edu} \\

  \end{tabular}}

\author{Vlad Roman and Robert M. Guralnick}

 \thanks{Both authors were partially supported by NSF grant DMS $1901595$.  
We thank  J.-P. Serre, Janos Kollar and Michael Larsen for helpful conversations regarding the last section. }

\title{The variety of nilpotent pairs $(A,B)$ with $[A,B] = \lambda I$ }

\date{\today}

\begin{document}

\begin{abstract}
    Let $k$ be an algebraically closed field of characteristic $p >0$. We consider the variety of nilpotent pairs $(A,B)$ with $[A,B]=\lambda I$, namely the set of pairs $ X = \{ (A,B) \in M_n(k) \times M_n(k) \mid A,B \text{ nilpotent}, [A,B]=AB-BA = \lambda I, \lambda \in k \}$. We prove that if $n=pr$, then $X$ is irreducible of dimension $n^2$.
\end{abstract}

\maketitle

\section{Introduction}
Let $k$ be an algebraically closed field and $\mathfrak{g}$ a Lie algebra over $k$. We call
$$\mathcal{C}_2(\mathfrak{g})=\{ (x,y) \in \mathfrak{g} \times \mathfrak{g} \mid [x,y]=0\}$$
the commuting variety of the Lie algebra $\mathfrak{g}$. The study of such varieties originates in a paper of Motzkin and Taussky \cite{Motzkin-Taussky}, who prove that $\mathcal{C}_{2}(\mathfrak{gl_n})$ is is an irreducible variety of dimension $n^2+n$ (for $k$ of any characteristic). This was later rephrased by Guralnick \cite{Guralnick} in terms of the Zariski topology and using denseness arguments. This paper also shows that an old result of Gerstenhaber \cite{Gerstenhaber} follows as an easy consequence of the irreducibility of $\mathcal{C}_{2}(\mathfrak{gl_n})$, namely a generalization of the Cayley-Hamilton theorem which asserts that any two generated commutative subalgebra of $M_n(k)$ has dimension at most $n$.

The problem was also studied by Richardson \cite{Richardson}, who proved that in characteristic $0$ and with the assumption that $\mathfrak{g}$ is reductive, $\mathcal{C}_2(\mathfrak{g})$ is irreducible and moreover, one has $\dim(\mathcal{C}_2(\mathfrak{g})) = \dim(\mathfrak{g}) + \rank(\mathfrak{g})$, where $\rank(\mathfrak{g})$ is the dimension of a maximal toral subalgebra. A paper of Levy \cite{Levy} extends Richardson's theorem to a field $k$ of good characteristic, that is $char(k) \neq 2$ for corresponding Lie groups of type $B, C, D$, $char(k) \neq 2,3$ for $G_2, F_4, E_6, E_7$ and $char(k) \neq 2, 3, 5$ for $E_8$.

A classical variation on the problem is considering commuting nilpotent pairs, namely the variety $\mathcal{C}_2^{nil}(\mathfrak{g}) = \{ (x,y) \in \mathfrak{g} \times \mathfrak{g} \mid [x,y]=0, x,y \text{ nilpotent}\}$. Due to papers by Kirillov and Neretin \cite{Kirillov-Neretin} and Šivic \cite{Sivic1},\cite{Sivic2},\cite{Sivic3}, it is known that for $n \leq 3$, $\mathcal{C}_r^{nil}(\mathfrak{gl_n})$ is irreducible for any $r$. This was generalized by Premet \cite{Premet}, who proved that for $\mathfrak{g}$ semisimple, $\mathcal{C}_2^{nil}(\mathfrak{g})$ is an equidimensional variety and its irreducible components are parametrized by the distinguished nilpotent orbits in $\mathfrak{g}$ (i.e. the orbits of nilpotent elements which are not centralized by a nontrivial torus in the algebra group).  In particular, $\mathcal{C}_2^{nil}(\mathfrak{gl_n})$ is irreducible for any $n$ in any characteristic. The $\mathfrak{gl}_n$ result is equivalent to the  irreducibility of a Hilbert scheme associated to $\mathbb{P}^1$ with three points removed.  See \cite{Baranovsky, Premet}.   For general $r$, the variety of commuting $p$-nilpotent $r$-tuples has connections to the representation theory of the $r$th Frobenius kernel of the Lie algebra.   See \cite{FriedlanderPevtsova, FriedlanderPevtsovaSuslin} for more details about this.  

One can also study the variety of commuting triples $\mathcal{C}_3(\mathfrak{g})=\{ (x_1,x_2,x_3) \in \mathfrak{g}^3 \mid [x_i,x_j]=0\}$. Due to Guralnick's paper \cite{Guralnick} and an improvement by Holbrook and Omladič \cite{Holbrook}, it is known that $\mathcal{C}_3(\mathfrak{gl_n})$ is reducible for $n \geq 29$.  See \cite[Section 7.9]{OCV}. The analogue of the Cayley-Hamilton theorem for $3$-generated algebras is still open (but does hold for $n \le 11$ by irreducibility results).
Gerstenhaber \cite{Gerstenhaber} proved the reducibility of $\mathcal{C}_r(\mathfrak{gl_n}), r, n \ge 4$ by observing
that there are $4$-generated commutative subalgebras of $M_n(k)$ that have dimension greater than $n$.

There has been little study for the case of bad characteristic and in type A, of not very good characteristic (i.e. the case of the characteristic 
dividing $n$ for $\mathfrak{gl_n}$).   The first author \cite{Ro1}  considered the case of $\mathfrak{gl_n}(k)$ with $p$ dividing $n$, showing that there are two irreducible components
(and giving the dimension of each component).   The case of bad characteristic will be studied in \cite{Ro3}.  

 In this paper we consider nilpotent pairs whose commutator is a multiple of the identity, that is the variety 
$$X = \{ (A,B) \in M_n(k) \times M_n(k) \mid A,B \text{ nilpotent}, [A,B]=\lambda I, \lambda \in k \},$$
where $k$ is an algebraically closed field of characteristic $p>0$ with $p$ dividing $n$.   If we fix $\lambda$, 
set $X_{\lambda} :=\{ (A,B) \in M_n(k) \times M_n(k) \mid A,B \text{ nilpotent}, [A,B]=\lambda I, \lambda \in k \}$.  Note
that these are all isomorphic for $\lambda \ne 0$ and $X_1$ is precisely the set of pairs $\phi(x), \phi(y)$ of the first Weyl
algebra where $\phi$ is an algebra homomorphism  with $\phi(x)$ and $\phi(y)$ nilpotent. This will be discussed in more detail in Section \ref{prelim}.

We prove the following.

\begin{theorem}
 $X$ is an irreducible variety of dimension $n^2$.
\end{theorem}

Of course, this gives the corresponding statement for the commuting variety of $\mathfrak{pgl_n}(k)$.  

The two key ingredients in the proof are the representation theory of the first Weyl algebra in characteristic $p$ and Premet's result \cite{Premet} on
the irreducibility of the commuting variety of nilpotent pairs in $\mathfrak{gl_n}(k)$.  This was originally proved in sufficiently large characteristic
by Baranovsky \cite{Baranovsky} using the irreducibility of a certain Hilbert scheme. In the last section we give a proof showing that many results about the components of the commuting variety  can be obtained in good characteristic from the corresponding result in characteristic $0$.

We use some basic results from algebraic geometry and the Zariski topology which for example can be found in \cite{Hartshorne1977}. See the beginning of the next section.

\section{Preliminary Results}\label{prelim}

We first record the following elementary and well known facts.

\begin{fact}[Fibre dimension theorem]
    Let $X,Y$ be varieties over the algebraically closed field $k$ with $X$ irreducible.  Let  $f: X \rightarrow Y$ be a dominant morphism. 
    If $y \in f(X)$, then  every component of $f^{-1}(y)$ has dimension at least $\dim X - \dim Y$.  
    \end{fact}
    
\begin{fact} Let $X$ be a closed subariety of projective space.  If $f:X \rightarrow Y$ is a morphism, then $f(X)$ is a closed subvariety of $Y$.
 \end{fact}
 

We next make a  few observation about the first Weyl algebra, which will be useful later. 
Let $k$ be an algebraically closed field of characteristic $p>0$ and let $W=k[x,y]/(xy-yx-1)$ be the first Weyl algebra. We show some facts about this.

By \cite{Revoy} we have the following.
\begin{proposition}
    The centre $Z(W) = k[x^p,y^p]$.
\end{proposition}

\begin{proposition}
    $W$ is a free $Z(W)$-module of rank $p^2$.
\end{proposition}

\begin{proof}
    Indeed, a $Z(W)$-basis of $W$ is given by $\{ x^i y^j \mid 0 \leq i,j \leq p-1\}$.
\end{proof}

The next result is critical to our results and shows that knowing about (nilpotent) representations of the center of the first Weyl algebra classifies the
(nilpotent) representations of the algebra.  We can then use Premet's result.

Now let $f:W \rightarrow R$ be a finite-dimensional homomorphic image of $W$ such that $f(x),f(y)$ are nilpotent in $R$, say $f(x)^{p^a}=f(y)^{p^a}=0$.

\begin{proposition}
    $R \cong M_p(Z(R))$.
\end{proposition}

\begin{proof}
    Let $I$ be the ideal of $Z(W)$ generated by $x^{p^a}$ and $y^{p^a}$ . Consider $$\Tilde{R}=W/IW,$$
    which is still a finite dimensional $k$-algebra. 
    Since $\Tilde{R}$ is a finite dimensional $k$-algebra, $\Tilde{R}$ is Artinian. 
    Since $W$ (with $x,y$ nilpotent) has a unique simple module, so does $\Tilde{R}$.
    Therefore, $\Tilde{R}$ is a local ring with unique maximal (nilpotent) ideal $J(\Tilde{R})$ and we have $\Tilde{R}/J(\Tilde{R}) \cong M_p(k)$. 
    Let $S$ be the unique simple module of $\Tilde{R}$.
    
    Let $P$ be the projective cover of $S$, so we have that $P/J(P) \cong S$ is simple. By using the decomposition of $\Tilde{R}$ (as a $k$-algebra) into projective indecomposable modules, we obtain 
    $$\Tilde{R} \cong \underbrace{P \oplus \cdots \oplus P}_\text{$p$ copies}$$
    and $P$ is projective indecomposable and in fact $P$ is free of rank $p$ since $\Tilde{R}$ is local.

    Next, since $W$ is free over $Z(W)$ of rank $p^2$, we obtain that $\Tilde{R}=W/IW$ is free over $Z(W)/I$ of rank $p^2$. We can embed $\Tilde{R}$ into the endomorphism ring $\mathrm{End}_{Z/I}(P)$, so we have
    $$\Tilde{R} \xhookrightarrow{} \mathrm{End}_{Z(W)/I}(P) \cong M_p(Z(W)/I).$$
    Note that $\dim_k \Tilde{R} = p^2 \dim_k Z(W)/I = p^2 e.$
    On the other hand, $\dim M_p(Z(W)/I) = p^2 \dim_k Z(W)/I$, so it follows that $\Tilde{R} \cong M_p(Z(W)/I)$.

    Finally, since $\Tilde{R} \twoheadrightarrow R$, $R \cong \Tilde{R}/LR$ for some ideal $L$ of $Z(W)$. Therefore, $$\Tilde{R} \cong M_p(Z(W)/L) \cong M_p(Z(\Tilde{R})).$$
\end{proof}

\section{The variety}

Let us first fix some notation. We have the following sets:
$$X = \{ (A,B) \in M_n(k) \times M_n(k) \mid A,B \text{ nilpotent}, [A,B]=\lambda I, \lambda \in k \},$$
$$Y= \{ (A,B) \in M_n(k) \times M_n(k) \mid A,B \text{ nilpotent}, [A,B]=\lambda I \text{ for
some } \lambda \ne 0\},$$
$$\mathcal{D} = \{ (A,B) \in M_n(k) \times M_n(k) \mid A,B \text{ nilpotent, } [A,B]=0 \}.$$
Note that $\mathcal{D}$ is the commuting variety of nilpotent pairs, which is irreducible of dimension $n^2-1$ by \cite{Premet} and $Y$ is the open subvariety of $X$ with $\lambda \neq 0$.

Recall that by the Morita equivalence, any ring $R$ is Morita equivalent to the matrix ring $M_n(R)$. This allows us to identify any $R$-module $V$ as column vectors of size $p$ with entries in $M$, where $M$ is a $Z(W)$-module. The identification $V \leftrightarrow M$ is an equivalence between the categories of $R$-modules and $M_n(R)$-modules. Using this fact we obtain the following:

\begin{proposition}\label{Morita}
    Two pairs $(A,B)$ and $(C,D)$  in $Y$ are conjugate  if and only if $(A^p,B^p)$ and $(C^p,D^p)$ are conjugate.
\end{proposition}

\begin{proof}
    By the Morita equivalence, the representations are equivalent as $R$-modules if and only if they are as $Z(W)$-modules.
\end{proof}

\begin{remark}
    $\mathcal{D}$ is contained in $\overline{Y}$.
\end{remark}
To see this, take $N$ to be a regular (i.e. nonderogatory) nilpotent matrix with $[N,B] = I$ and notice that 
$$[N,bB+f(N)] = [N,bB] + [N,f(N)] = b[N,B] + 0 = bI,$$
so $(N,bB+f(N)) \in Y$. Thus,     $\overline{Y}$ (the Zariski closure of $Y$)  contains all pairs of the form $(N,f(N))$, so all regular commuting nilpotent pairs. By \cite{Premet}, this is open and dense in $\mathcal{D}$.

An easy consequence of the Morita correspondence is the following: 

\begin{lemma}\label{CD=DC}   If $(A,B) \in Y$, then 
  $(A,B)$  is conjugate to a pair of block diagonal matrices with all matrix blocks of size $r \times r$, where $n = pr$ and $(A^p,B^p) = ((C, ..., C), (D, ..., D))$ with $CD=DC$.
\end{lemma}

\begin{proof}
    Let $(A,B) \in Y$. Consider the module $P$ for any finite-dimensional homomorphic image $R$ of $W$ of dimension $n$. Then
$\langle A,B \rangle$ generate $M_p(Z)$, where $Z =k[A^p, B^p]$ and so we can conjugate $R$ to
equal $M_p(Z)$ corresponding to block $r \times r$ matrices with $Z$ corresponding to scalar matrices.   The result follows. 
\end{proof}

Now consider the map $\phi$ defined on $Y$  sending $(A,B) \mapsto (A^p,B^p)$.    

\begin{proposition}\label{image_irreducible}
    The image $\phi(Y) = \{ (A^p,B^p) | (A,B) \in Y\}$ is an irreducible variety.
\end{proposition}

\begin{proof}  Consider the closure $Y'$  of all pairs in $Y$ that are of the form of Lemma \ref{CD=DC}.   Note
that we can view this as a homogenous subvariety of pairs of matrices.   Thus,  the image of $\phi(Y')$ is closed.
Since there exists $(A,B) \in Y$ with $A$ regular nilpotent \cite{Ro1} each block of $A^p$ is a regular
$r \times r$ nilpotent matrix.  Note that
 $\phi(Y')$ is closed under conjugation
by block diagonal matrices which are constant on the diagonal and $0$ off the diagonal.  Thus, $\phi(Y')$ contains the closure
of all nilpotent commuting pairs $(C,D)$ with $C$ regular nilpotent.  By Premet's result, we see that $\phi(Y')$ is isomorphic
to the variety of commuting nilpotent pairs of $n \times n$ matrices and in particular is irreducible.
\end{proof}

We also recall from \cite{Ro1}, that if $A \in M_n(k)$ is nilpotent, then there exists $B \in M_n(k)$ nilpotent with
$[A,B] = \lambda I$ for some nonzero $\lambda$  if and only if all Jordan blocks of $A$ have size a multiple of $p$.

\begin{lemma}\label{fixedconjclass}
    Let $E$ be a conjugacy class of nilpotent matrices in $M_n(k)$ with all Jordan blocks of size a multiple of $p$ and let 
    $$\mathcal{V} =  \mathcal{V}_E = \{ (A,B)  \in E \times M_n(k) \mid B \ \text{nilpotent and}  \ [A,B]=\lambda I \text{ for
some } \lambda \in k^*\}.$$
Then $\mathcal{V}$ is an irreducible variety.
\end{lemma}

\begin{proof}
    By conjugating we may assume without loss of generality that $A,B$ are block upper triangular, so of the form
    $$\begin{pmatrix}
A_1 & * & *\\
 & \ddots & * \\
 &  & A_l
\end{pmatrix}, 
\begin{pmatrix}
B_1 & * & *\\
 & \ddots & * \\
 &  & B_m
\end{pmatrix}.$$

Next, by Lemma $2.3$ in \cite{Ro1} and by conjugating on each block, it follows that the pair is similar to 

$$\begin{pmatrix}
A_0 & * & *\\
 & \ddots & * \\
 &  & A_0
\end{pmatrix}, 
\begin{pmatrix}
B_0 & * & *\\
 & \ddots & * \\
 &  & B_0
\end{pmatrix},$$
with $A_0,B_0$ corresponding to the $p \times p$ case.

Let $C \in M_n(k)$ centralize $A$ and be of the form 
$$C = \begin{pmatrix}
0 & * & *\\
 & \ddots & * \\
 &  & 0
\end{pmatrix}.$$
Then note that $[A,B+C] = [A,B] + [A,C] = \lambda I$, so the pair $(A,B+C) \in \mathcal{V}$.
Next, since the space $\{ (A,B+C)\}$ described above is a linear space (in particular it is irreducible) and there is a dominant map $GL_n(k) \times \mathcal{V} \rightarrow \mathcal{V}$ given by conjugating $(g,A,B) \mapsto (gAg^{-1},gBg^{-1})$, it follows that $\mathcal{V}$ is irreducible.
\end{proof}

\begin{lemma}\label{naturalmodulecyclic}
    Let $L \in M_r(k)$ be nilpotent. Then there exists nilpotent $M \in M_r(k)$ such that $T=k[L,M]$ is a commutative algebra and the natural module $V$ is a cyclic $T$-module.
\end{lemma}

\begin{proof}
    Consider the natural module $V$  for $L$,
    $$V_L \cong k[x]/(x^{r_1}) \oplus  k[x]/(x^{r_2}) \oplus \cdots \oplus  k[x]/(x^{r_s}),$$
    with $r_1 + r_2 + \dots + r_s = r$ and assume $r_1 \geq r_2 \geq \dots \geq r_s$.
    Also let $e_1,e_2,\dots,e_s$ be generators of the corresponding cyclic modules in the above decomposition. Choose $M$ to be the unique matrix that commutes with $L$ and maps  $e_i \mapsto e_{i+1}$ for all $i=1,\dots,s-1$. 
    Then $V/LV$ is generated by the images of the $e_i$ in the quotient and for the above chosen $M$, $V/LV$ is cyclic.  Then $T$ is a commutative $k$-algebra and $V$ is a free rank one $T$-module whence the result follows.  
\end{proof}

\begin{lemma}  Let $n=pr$.  Let $A \in M_n(k)$ be a nilpotent matrix with all Jordan blocks of size a multiple of $p$.  Let $0 \ne \lambda \in k$.
    Then there exists a nilpotent matrix $B \in M_n(k)$ such that $[A,B] = \lambda I$ and $\dim k[A^p,B^p]=r$. 
\end{lemma}

\begin{proof}  As in the proof of Proposition \ref{image_irreducible}, we work in the variety $Y'$.  Then the image of $\phi$ contains $(A^p, B^p)$ which
have constant diagonal for any such $B^p$ commuting with $A^p$.  In particular, we can choose $B$ so that $k[A^p, B^p]$ (viewed in $M_r(k)$)
is self-centralizing (in $M_r(k)$) of dimension $r$.  
     \end{proof}

\begin{definition}
    Call a pair of nilpotent elements $(A,B) \in M_n(k) \times M_n(k)$ \textbf{good} if $[A,B] = \lambda I, \lambda \ne 0$ and $\dim k[A^p,B^p]=r$, where $n=pr$.
\end{definition}

We remark that the condition $\dim k[A^p,B^p]=r$ is an open condition (since by Lemma 1,  $k[A^p, B^p]$ is a commutative $2$-generated subalgebra
of $M_r(k)$ and so $\dim k[A^p, B^p] \le r$).    If $E$ is the conjugacy class of $A$, since $E$ is irreducible, we see that the set of good 
$(A,B) \in \mathcal{V}_E$ is an open Zariski dense subset of $\mathcal{V}_E$.

\begin{lemma} \label{fiber}   Let $(A,B) \in Y$ be good.   Then $\phi^{-1}(A^p,B^p)$ is an irreducible  closed subvariety of  $Y$ of dimension 
$p^2r- r$.
\end{lemma}

\begin{proof}  Suppose $(A_0, B_0) \in Y$ with $\phi(A_0,B_0)=\phi(A,B)$.   By the Morita correspondence, $(A_0,B_0)$ is conjugate to $(A,B)$
and so necessarily by an invertible element of the centralizer of $k[A^p, B^p]$.  This is irreducible and so its orbit is as well and so the fiber
is irreducible.  The dimension is the difference of the dimension of the centralizer of $k[A^p,B^p]$ and the centralizer of $k[A,B]$.   The latter 
is $k[A^p,B^p]$ (because this is the center of $k[A,B] \cong M_p(k[A^p,B^p]$).  This gives the formula for the dimension of the orbit.  
\end{proof}    

\noindent
{\bf Proof of Theorem 1}

 Let $\mathcal{V}_{reg}$ denote the closure of $\mathcal{V}_E$ with $E$ the conjugacy class of regular nilpotent matrices.   As we have already noted,
 $\phi(\mathcal{V}_{reg}) = \phi(\bar{Y})$.   Any $(A,B) \in Y$ with $A$ regular nilpotent is good and so we see the generic fiber for the 
 map $\phi: \mathcal{V}_{reg} \rightarrow \phi(Y)$ has dimension $p^2r-r$ and so every fiber has at least that dimesion.
 If $(A,B)$ is good, then the full inverse image of $(A^p, B^p)$ is irreducible of dimension $p^2r-r$ and so this inverse image must be contained in
 $\mathcal{V}_{reg}$.   Since the good pairs are dense in $Y$,   $\mathcal{V}_{reg} $ contains $Y$ and since it is closed, it equals  
 $X$ and so $X$ is irreducible.  
 
 Finally we compute the dimension of $X$.   Consider the projection from $X$ to the set of nilpotent matrices.   This is a surjection (as $X$ contains
 the subvariety of commuting pairs) and so the image of this projection has dimension $n^2 - n$.  The generic fiber has dimesion $n$
 since generically the projection is a regular nilpotent element and for $A$ regular nilpotent the set of nilpotent $B$ with
 $[A,B]=\lambda I$ is a variety of dimension $n-1$ for a fixed $\lambda$ and so the full fiber has dimension $n$.  Thus, 
 $X$ is irreducible of dimension $n^2$.  This proves Theorem 1.

\section{From characteristic zero to positive characteristic}

In this section we give a new proof of Premet's result \cite{Premet} about the irreducibility of the variety of commuting
nilpotent pairs in $\mathfrak{gl_n}(k)$ for $k$ any algebraically closed field.  We show that the result follows
from the corresponding result in characteristic zero.  Premet gave an easier proof in characteristic zero but this had already
been proved by Baranovosky \cite{Baranovsky}, who showed that it followed by the irreducibility of a certain Hilbert scheme.
A consequence of Premet's result was the irreducibility of the corresponding Hilbert scheme in all characteristics.

Indeed, we illustrate that these methods show that many results about the structure of commuting varieties and commuting
nilpotent varieties in good characteristic follow from the corresponding characteristic zero results.  

The key ideas in the proof are the Principal Ideal Theorem for commutative rings \cite{Hartshorne1977} and the structure of the centralizer of an element $x \in \mathfrak{g}$ where $\mathfrak{g}$ is a simple Lie algebra in good characteristic.

We first give an elementary result in commutative algebra.  In fact, a stronger statement is true but
we do not need it.

\begin{lemma} \label{affine}  Let $A$ be a finitely generated $\Z$-algebra and assume that $A$ is torsion-free.
Let $p$ be a prime which is not invertible in $A$.
Let $d$ be the minimal length of a maximal chain of prime ideals of $A$.
\begin{enumerate}
\item  There is a maximal chain of primes ideals of length less than $d$ in $B:= A \otimes \Q$.
\item    Every maximal chain of prime ideals in $A/pA$ has length at least $d-1$.
\end{enumerate}
\end{lemma}

\begin{proof}  Without loss of generality, we may assume that  $A$ is an integral domain (quotient by
the minimal prime in the appropriate chain) and so $A$ has Krull dimension $d$.   If 
$Q_0 = 0 < Q_1 < \ldots < Q_r$ is a maximal chain of ideals in $B$, then $Q_i \cap A \le Q_{i-1} \cap A$
and so $0 < Q_1 \cap A  < \ldots < Q_r \cap A$.   By the Nullstellensatz,  $Q_r \cap A$ is not a maximal ideal
and so $r < d$ as claimed.

The prime ideals of $A/pA$ corresponds to the prime ideals of $A$ containing $p$.   By the principal ideal theorem \cite{Hartshorne1977},
any prime minimal over $pA$ has  height at most $1$ and so a maximal chain in $A$ in which the minimal prime contains
$A$ gives rise to a maximal chain in $A/pA$ of length at least $d-1$.
\end{proof}

We apply this in the following setup.   Let $X$ be a variety defined over $\Z$ and let $I$ be the ideal defining this variety and assume
that $A=\mathbb{Z}[X]/I$ is torsion-free.
Let $J$ be the reduced ideal defining the variety over $\Q$ (so $I < J$).    Let $B = \Q[X] = \Q[x_i]/J$ and set $A= \Z[X]/(J \cap \Z[x_i])$.
Applying the previous result implies that every maximal chain of prime ideals in $A/pA$ has length at least that of the minimal such
for $B$ (indeed one sees that if $B$ has all maximal chains of the same length, the same is true for $A/pA$ -- we will not use this).

The variety corresponding to $A/pA$ is contained in the variety over the field corresponding to $I$ and every component 
of this variety has dimension at least the dimension of the smallest irreducible component of $X$.
 One issue that arises is to identify the variety associated to $A/pA$.  

We use this to give a proof of Premet's result about commuting nilpotent pairs in simple Lie algebras in good characteristic by showing it suffices
to know the result in characteristic zero.   

Fix a prime $p$.  Let $\mathfrak{g}$ be a simple Lie algebra and assume that $p$ is good for $\mathfrak{g}$.   To simply the proof, we assume that
$\mathfrak{g}$ is of type $A, B, C$ or $D$ (and so $p$ is odd for the latter three cases).   Let $G$ be the corresponding algebraic group.  

Let $A$ be the affine $\Z$-algebra so that the corresponding $\Q$-algebra corresponds to the variety $X$ of commuting nilpotent pairs. 

 Let $K$ be the algebraic closure of $\Q$ and $k$ the algebraic closure of the field of $p$-elements.  Let $X(K)$ be the variety
of commuting nilpotent pairs of elements in $\mathfrak{g}(K)^2$ and $X(k)$ the variety
of commuting nilpotent pairs of elements in $\mathfrak{g}(k)^2$.   Let $\mathcal{O}$ denote the ring of all algebraic integers in $K$ and let
$\mathcal{P}$ be a maximal ideal of $\mathcal{O}$ containing $p$.

Note that every line in $X(K)$ intersects $X(\mathcal{O})$.   Let $N \in \mathfrak{g}(k)$ be nilpotent.  Then we can choose $N$ to be in Jordan form. 
Let  $N' \in \mathfrak{g}(\mathcal{O})$ also be in Jordan canonical form so that $N$ is the  reduction of $N'$  modulo $\mathcal{P}$.  
The structure of the centralizer in $\mathfrak{g}$ of $N$ (and $N'$) is well known and it clear that the integral points of the centralizer of $N'$ surject onto
the $k$-points of the centralizer of $A$.   It is straightforward to show that any nilpotent element of the centralizer of $N$ lifts to an integral nilpotent
element of the centralizer as well of $N'$.   Thus, any pair with first coordinate $N$ lifts to a point of $X(\mathcal{O})$. 
 Since we can conjugate by $G(\mathcal{O})$ which surjects onto $G(k)$, it follows that $X(\mathcal{O})$ surjects onto $X(k)$. 
 This shows that $X(k)$ is the variety of commuting nilpotent pairs.     
 
 Note that in bad characteristic, there are extra classes and moreover, even for classes that do correspond, centralizers can be larger.
 Thus the proof (and result) fail in that case.

 By Lemma \ref{affine},  every component of $X(k)$ has dimension at least that of $X(K)$ (and indeed we have equality).   
 Since there are only finitely many classes of nilpotent elements,
 any component of $X(k)$ would have an open dense subset of the form $(C,D)$ with $C$ in a fixed conjugacy class and $D$ a nilpotent element
 of the centralizer of $C$.   Since we are in good characteristic,  the dimension of the centralizer of $D$ in $\mathfrak{g}$ is the same as the dimension
 of the centralizer in the corresponding algebraic group.  Thus, the dimension of this component is at most $\dim \mathfrak{g}$ with equality if and only if
 the centralizer of $C$ consists only of nilpotent elements, i.e. the class of $C$ is distinguished.   It follows that the components of $X$ are in bijection
 with the classes of distinguished nilpotent elements, as stated by Premet.
 
 The same proof applies for the exceptional Lie algebras in good characteristic using the characterization of centralizers given in \cite{liebeck-seitz}.

 A similar proof shows that the variety of commuting pairs in $\mathfrak{g}$ in good characteristic is irreducible of dimension equal to
 $\dim (\mathfrak{g}) + \rank(\mathfrak{g})$ once one has the result for characteristic zero.   Motzkin and Taussky \cite{Motzkin-Taussky} proved the result in all characteristics
 for $\mathfrak{gl}_n$ and Richardson \cite{Richardson} proved this result in characteristic zero for all $\mathfrak{g}$.  Thus, Richardson's result implies the result in good
 characteristic.  This was proved by Levy \cite{Levy}.    Indeed, the result is that the variety of commuting pairs in $\mathfrak{g}$ is the closure of the union of
 all conjugates of $\mathfrak{t} \times \mathfrak{t}$ where $\mathfrak{t} = \mathrm{Lie}(T)$ for $T$ a maximal torus of the algebraic group.  Induction on
 dimension shows that if the result fails, there would be a component consisting of commuting nilpotent pairs. The dimension of all commuting nilpotent pairs has dimension equal to $\dim (\mathfrak{g})$, but the result above (using Richardon's result in characteristic zero)  implies that the dimension of any component is at least $\dim (\mathfrak{g}) + \dim (\mathfrak{t})$ and the result holds. 
 
We note also that the result for nilpotent pairs implies the result for all commuting pairs.

\printbibliography

@article{Motzkin-Taussky,
   author = {Motzkin, T. and Taussky, O.},
    title = "{Pairs of matrices with property L. II}",
  journal = {Trans. Amer. Math. Soc. 80, 387-401},
     year = 1955
}

@article{Gerstenhaber,
  title={On Dominance and Varieties of Commuting Matrices},
  author={Gerstenhaber, M.},
  journal={Annals of Mathematics},
  year={1961},
  volume={73},
  pages={324}
}

@article{Richardson,
author = {Richardson, R. W.},
journal = {Compositio Mathematica},
language = {eng},
number = {3},
pages = {311-327},
publisher = {Sijthoff et Noordhoff International Publishers},
title = {Commuting varieties of semisimple Lie algebras and algebraic groups},
url = {http://eudml.org/doc/89407},
volume = {38},
year = {1979},
}

@article{Kirillov-Neretin,
  title={The variety $A_n$ of n-dimensional Lie algebra structures},
  author={Alexander A. Kirillov and Yury A. Neretin},
  journal = {Trans. Amer. Math. Soc. 137, Vol. 2},
  year={1987}
}

@article{Premet,
author = {Premet, Alexander},
year = {2003},
month = {03},
pages = {},
title = {Nilpotent commuting varieties of reductive Lie algebras},
volume = {154},
journal = {Inventiones Mathematicae},
doi = {10.1007/s00222-003-0315-6}
}

@article{Levy,
title = {Commuting Varieties of Lie Algebras over Fields of Prime Characteristic},
journal = {Journal of Algebra},
volume = {250},
number = {2},
pages = {473-484},
year = {2002},
issn = {0021-8693},
doi = {https://doi.org/10.1006/jabr.2001.9083},
url = {https://www.sciencedirect.com/science/article/pii/S0021869301990830},
author = {Paul Levy}
}

@article {Revoy,
    AUTHOR = {Revoy, Philippe},
     TITLE = {Alg\`ebres de {W}eyl en caract\'{e}ristique {$p$}},
   JOURNAL = {C. R. Acad. Sci. Paris S\'{e}r. A-B},
  FJOURNAL = {Comptes Rendus Hebdomadaires des S\'{e}ances de l'Acad\'{e}mie des
              Sciences. S\'{e}ries A et B},
    VOLUME = {276},
      YEAR = {1973},
     PAGES = {A225--A228},
      ISSN = {0151-0509},
   MRCLASS = {16A16},
  MRNUMBER = {335564},
MRREVIEWER = {Lindsay N. Childs},
}

@article{Holbrook,
title = {Approximating commuting operators},
journal = {Linear Algebra and its Applications},
volume = {327},
number = {1},
pages = {131-149},
year = {2001},
issn = {0024-3795},
doi = {https://doi.org/10.1016/S0024-3795(00)00286-X},
url = {https://www.sciencedirect.com/science/article/pii/S002437950000286X},
author = {John Holbrook and Matjaž Omladič},
}

@article{Guralnick,
  title={A note on commuting pairs of matrices},
  author={Robert M. Guralnick},
  journal={Linear \& Multilinear Algebra},
  year={1992},
  volume={31},
  pages={71-75}
}

@article{Sivic1,
title = {On varieties of commuting triples},
journal = {Linear Algebra and its Applications},
volume = {428},
number = {8},
pages = {2006-2029},
year = {2008},
issn = {0024-3795},
doi = {https://doi.org/10.1016/j.laa.2007.11.004},
url = {https://www.sciencedirect.com/science/article/pii/S0024379507005228},
author = {Klemen Šivic}
}

@article{Sivic2,
title = {On varieties of commuting triples II},
journal = {Linear Algebra and its Applications},
volume = {437},
number = {2},
pages = {461-489},
year = {2012},
issn = {0024-3795},
doi = {https://doi.org/10.1016/j.laa.2011.08.014},
url = {https://www.sciencedirect.com/science/article/pii/S0024379511005970},
author = {Klemen Šivic}
}

@article{Sivic3,
title = {On varieties of commuting triples III},
journal = {Linear Algebra and its Applications},
volume = {437},
number = {2},
pages = {393-460},
year = {2012},
issn = {0024-3795},
doi = {https://doi.org/10.1016/j.laa.2011.08.015},
url = {https://www.sciencedirect.com/science/article/pii/S0024379511005982},
author = {Klemen Šivic}
}

@article{Ro1,
title = {The commuting variety of pgln},
journal = {Journal of Algebra},
volume = {665},
pages = {229-242},
year = {2025},
issn = {0021-8693},
url = {https://www.sciencedirect.com/science/article/pii/S0021869324005945},
author = {Vlad Roman},
}

@article{Ro3,
title = {The commuting variety of Lie algebras in bad characteristic},
author = {Vlad Roman},
journal = {in preparation},
year = {2025}
}

@article{Baranovsky,
title = {The variety of pairs of commuting nilpotent matrices is irreducible},
journal = {Transformation Groups},
volume = {6},
year = {2001},
doi = {https://doi.org/10.1007/BF01236059},
author = {Baranovsky, V.}
}

@book{liebeck-seitz,
  title={Unipotent and Nilpotent Classes in Simple Algebraic Groups and Lie Algebras},
  author={Liebeck, M.W. and Seitz, G.M.},
  isbn={9780821869208},
  lccn={2011043518},
  series={Mathematical surveys and monographs},
  url={https://books.google.ro/books?id=Th-CAwAAQBAJ},
  year={2012},
  publisher={American Mathematical Society}
}

@book{Hartshorne1977,
  author    = {Hartshorne, Robin},
  title     = {Algebraic Geometry},
  series    = {Graduate Texts in Mathematics},
  volume    = {52},
  publisher = {Springer-Verlag},
  year      = {1977},
  isbn      = {978-0-387-90244-9},
  address   = {Berlin, New York},
}

@article{FriedlanderPevtsovaSuslin,
    author = {Friedlander, Eric M. and Pevtsova, Julia and Suslin, Andrei},
    title = {Generic and maximal Jordan types},
    journal = {Inventiones mathematicae},
    year = {2007},
    volume = {168},
    pages = {485-522},
    doi = {https://doi.org/10.1007/s00222-007-0037-2}
}

@article{FriedlanderPevtsova,
author = {Friedlander, Eric and Pevtsova, Julia},
year = {2011},
month = {06},
pages = {},
title = {Generalized support varieties for finite group schemes},
journal = {Doc. Math.}
}

@book{OCV,
  title={Advanced Topics in Linear Algebra: Weaving Matrix Problems Through the Weyr Form},
  author={O'Meara, K. and Clark, J. and Vinsonhaler, C.},
  isbn={9780199793730},
  lccn={2011003565},
  url={https://books.google.com/books?id=lXoFTvPD5MsC},
  year={2011},
  publisher={Oxford University Press, USA}
}
 
\end{document}